%% file: arxiv-Some-remarks-on-Einstein-Randers-metrics.tex
\documentclass[10pt,a4paper]{article}
\usepackage{amsmath}
\usepackage{latexsym,amssymb,amsmath,amsthm,amsfonts}
\usepackage{graphicx}
\usepackage{enumerate}
\usepackage{color}

\allowdisplaybreaks[4]

\newtheorem{lemma}{Lemma}[section]
\newtheorem{proposition}[lemma]{Proposition}
\newtheorem{theorem}[lemma]{Theorem}

\newtheorem{example}[lemma]{Example}
\newenvironment{remark}{\textbf{Remark}}{}

\numberwithin{equation}{section}

\input{symbol}

\makeatletter

\newcommand{\Rmnum}[1]{\expandafter\@slowromancap\romannumeral #1@}
\makeatother

\begin{document}
\title{Some remarks on Einstein-Randers metrics}
\footnotetext{\emph{Keywords}:Randers metric, Einstein metric, flag curvature, navigation problem.
\\
\emph{Mathematics Subject Classification}: 53B40, 53C60.}

\author{Xiaoyun Tang and Changtao Yu}

\date{2017.06.01}
\maketitle

\begin{abstract}
In this essay, we study the sufficient and necessary conditions for a Randers metrc to be of constant Ricci curvature, without the restriction of strong convexity (regularity). A classification result for the case $\|\b\|_\a>1$ is provided, which is similar to the famous Bao-Robles-Shen's result for strongly convex Randers metrics ($\|\b\|_\a<1$). Based on some famous vacuum Einstein metrics in General Relativity, many non-regular Einstein-Randers metrics are constructed. Besides, we find that the case $\|\b\|_\a\equiv1$ is very distinctive. These metrics will be called singular Randers metrics or parabolic Finsler metrics since their indicatrixs are parabolic hypersurfaces. A preliminary discussion for such metrics is provided.
\end{abstract}

\section{Introduction}
In the past fifteen years, one of the most inspiring progress in Finsler geometry is the classification of strongly convex Randers metrics with constant flag or Ricci curvature\cite{db-robl-szm-zerm,db-robl-onri}. A wonderful review article \cite{db-rsf} is highly recommended. Randers metrics are the simplest non-Riemann Finsler metrics expressed as $F=\a+\b$, where $\a$ is a Riemann metric and $\b$ is a $1$-form~(Randers's original idea employs Lorentz rather than Riemann metrics because he was operating within the context of General Relativity\cite{Ran,db-rsf}).

In short, the earliest characterisation of Randers space forms was provided by Yasuda-Shimada in 1977. However, Shen constructed many countable examples in 2001, which indicated that Yasuda-Shimada's result was incorrect. Meanwhile, Bao and Robles provided a modified characterisation, which leaded to the final classification finished by Bao, Robles and Shen. Their nice result says, a strongly convex Randers metric $F=\a+\b$ has constant flag curvature if and only if after some suitable deformations on $\a$ and $\b$, the resulting Riemann metric $\ba$ has constant sectional curvature, and $1$-form $\bb$ is homothetic with respect to $\ba$. Soon afterwards, Bao and Robles made clear the structure of Einstein-Randers metrics\cite{db-robl-onri}. See Theorem \ref{uabubdagnianengga} for details. The key technique relates to the classical Zermelo's navigation problem on Riemann spaces, which is also the key for Shen to construct his famous ``fish pond"\cite{shen-fmwk}.

It is known that a Randers metric $F=\a+\b$ is strongly convex (regular) if and only if $b:=\|\b\|_\a<1$. Recently, we find that the cases when $b\equiv1$ or $b>1$, especially the later, are also interesting. Although both of them as metrics have some singularity.

Specifically, the curvature structure of Randers metrics with $b>1$ is much similar to that of strongly convex Randers metrics. The final classification (Theorem \ref{uabubdagnianengga}) shows that a Randers metric $F=\a+\b$ with $b>1$ has constant flag curvature if and only if after some suitable deformations on $\a$ and $\b$, the resulting metric $\ba$ is a Lorentz metric with constant sectional curvature, and $1$-form $\bb$ is homothetic with respect to $\ba$. That is to say, such metrics have close relationship with Lorentz other than Riemann metrics. In this sense, they match Randers's original idea better. Perhaps they will be of some potential applications in physics in the future.

The case when $b\equiv1$ is exceptional and different form the case when $b\neq1$ essentially. Its curvature structure is complicated and hard to forecast. We will call such a metric as a \emph{singular Randers metric}. However, its singularity is negligible. It will also be called a \emph{parabolic Finsler metric}, since its indicatrix is parabolic hypersurface. See Section 4 for related discussions. The argument on singular Randers metrics in this essay is limited. In fact, the second author has a long-term project to make clear the curvature structure of thess metrics. We have obtain many encouraging results, although there are still many formidable obstacles.

In summary, this essay includes four essentials listed below.
\begin{itemize}
\item A series of priori formulae on $\b$ are summarized in Section 3. Parts of them had been proved and used effectively earlier, see \cite{db-robl-onri} or \cite{cxy-szm-fgaa}. These priori formulae will be useful for the further study, especially for Riemann metrics and general $\ab$-metrics\cite{yct-zhm-onan}.
\item A brief proof of characterization for Randers metrics of constant Ricci or flag curvature, without insisting on strongly convexity, is provided in Section 5 and 6 respectively.
\item A complete local classification for Randers metrics of constant Ricci or flag curvature with $b>1$ is obtained in Section 7. In Section 8 and Section 9 some typical examples are listed. They are constructed by using some well-known vacuum solutions of Einstein's field equations, such as Schwarzschild metric, Kerr metric, Kasner metric, etc.
\item Some discussions on particularity of singular Randers metrics are given in the last Section.
\end{itemize}

\section{Preliminaries}
A Finsler metric on a $n$-dimensional manifold $M$ is a smooth function $F(x,y)$ on entire slit tangent bundel $TM\backslash\{0\}$ with positive homogeneity of degree $1$, where $x$ denotes position of $M$ and $y$ denotes direction in $T_xM$. Moreover, nonnegativity of a metric requires $F$ to be positive on $TM\backslash\{0\}$. In order to calculate other geometric quantities such as Riemann curvature, $F$ is usually asked to be strongly convex. namely the $y$-Hessian $g_{ij}:=(\frac{1}{2}F^2)_{y^iy^j}$ is positive definite. Nonnegativity together with strongly convexity indicate that the indicatrix of $F$ at any point $x$ is a strongly convex hypersurface in $T_xM$.

However, the Riemann curvature tensor is well-defined so long as $g_{ij}$ is non-degenerate. Meanwhile, the restriction of nonnegativity can be relaxed too. The famous Kropina metric $F=\frac{\a^2}{\b}$ in Finsler geometry is a typical example.

Given a Finsler metric $F$,
\begin{eqnarray*}
G^{i}:=\frac{1}{4}g^{il}\left\{[F^{2}]_{x^{k}y^{l}}y^{k}-[F^{2}]_{x^{l}}\right\}
\end{eqnarray*}
are its spray coefficients, in which $(g^{ij})=(g_{ij})^{-1}$. The Riemann tensor ${}^FR_{y}=\RIkF\frac{\partial}{\partial x^{i}}\otimes
dx^{k}$ of $F$ is determined by Berwald's formula as follows,
\begin{eqnarray*}
\RIkF:=2G^{i}_{x^{k}}-G^{i}_{x^{l}y^{k}}y^{l}+2G^{l}G^{i}_{y^{k}y^{l}}-G^{i}_{y^{l}}G^{l}_{y^{k}}.
\end{eqnarray*}
$F$ is said to be of constant flag curvature $K$ if and only if $\RIkF=K(F^2\dIk-Fy^iF_{y^k})$. The Ricci curvature of $F$ is the trace of $(\RIkF)$. By the Ricci scalar ${}^FR^m{}_m$ one can define the Ricci tensor by
\begin{eqnarray*}
{}^F{Ric}_{ij}:=\left(\f{1}{2}{}^FR^m{}_m\right)_{y^iy^j}.
\end{eqnarray*}
This definition is due to Akbar-Zadeh. Hence, ${}^FR^m{}_m={}^F{Ric}_{ij}y^iy^j$ and we will denote the Ricci curvature by $\RicooF$ in this paper. $F$ is said to be of constant Ricci curvature $K$ if and only if $\RicooF=(n-1)KF^2$. In this case, $F$ is called an Einstein-Finsler metric with Ricci constant $K$.

We need some abbreviations. Let $\a=\sqrt{a_{ij}y^iy^j}$ be a Riemann metric, $\b=b_iy^i$ be a $1$-form. Denote
\begin{eqnarray*}
\rij:=\f{1}{2}(\bij+b_{j|i}),\quad\sij:=\f{1}{2}(\bij-b_{j|i})
\end{eqnarray*}
be the symmetrization and antisymmetrization of the covariant derivative $\bij$ respectively, then
\begin{eqnarray*}
&r_{i0}:=r_{ij}y^j,\quad r^i{}_0:=a^{im}r_{m0},\quad r_{00}:=r_{i0}y^i,\quad r_i:=b^mr_{mi},\quad r^i:=a^{im}r_{m},\quad r_0:=r_iy^i,\quad r:=r_ib^i,\\
&s_{i0}:=s_{ij}y^j,\quad s^i{}_0:=a^{im}s_{m0},\quad s_i:=b^ms_{mi},\quad s^i:=a^{im}s_{m},\quad s_0:=s_iy^i.
\end{eqnarray*}
Roughly speaking, indices are raised or lowered by $a^{ij}$ or $\aij$, vanished by contracted with $b^i$~(or~$\bi$) and changed to be '${}_0$' by contracted with $y^i$ (or $y_i:=a_{ij}y^j$). At the same time, we need other three tensors
\begin{eqnarray*}
p_{ij}:=r_{im}r^m{}_j,\qquad q_{ij}:=r_{im}s^m{}_j,\qquad t_{ij}:=s_{im}s^m{}_j
\end{eqnarray*}
and the related tensors determined by the above rules. Notice that both $p_{ij}$ and $t_{ij}$ are symmetric, but $q_{ij}$ is neither symmetric nor antisymmetric in general. So $b^jq_{ji}$ and $b^jq_{ij}$, denoted by $q_i$ and $q^\star_i$ respectively, are different. But $b^iq_i$, denoted by $q$, is equal to $b^iq^\star_i$. Finally, in order to avoid ambiguity, sometimes we will use index $b$, which means contracting the corresponding index with $b^i$ or $b_i$. For example, $\rohb=r_{0|k}b^k$.

\section{Some priori formulae on $\b$}
In this section, we summarize some criteria which a $1$-form $\b$ must satisfy on a Riemann space. They are listed in (\ref{3ingaindngiad})-(\ref{relation6}). The key formula (\ref{relation1}) play a crucial role in the study of strongly convex Randers metrics \cite{db-robl-onri}.

A remarkable feature of these formulae is that all of them are priori. That is to say, they hold for any Riemann metric $\a$ and any $1$-form $\b$ without any restriction. These formulae show the inner relationship between the covariant derivative of $\sij$ and the Riemann tensor. Hence, one can use them, at the very beginning of the whole discussion, to replace all the terms related to the covariant derivative of $\sij$ with the Riemann curvature. And likewise, one can use them as the final check, just as we do in Section 5. In Section 6 our approach is a little different. we use them in the middle of the discussions. Anyway, these priori formulae will be very useful for the further research, especially for the related topics on Riemann metrics and general $\ab$-metrics\cite{yct-zhm-onan}.
\begin{proposition}
\begin{eqnarray}\label{3ingaindngiad}
r_{i|j}+s_{i|j}=r_{j|i}+s_{j|i},
\end{eqnarray}
as a corollary,
\begin{eqnarray}
\sohb=-p_{0}-q_{0}+q^{*}_{0}-t_{0}-\rohb+r_{|0}.\label{relation5}
\end{eqnarray}
\end{proposition}
\begin{proof}
(\ref{3ingaindngiad}) is a direct result of the fact $(b^2)_{|i|j}=(b^2)_{|j|i}$. Contracting (\ref{3ingaindngiad}) with $y^ib^j$ yields (\ref{relation5}).
\end{proof}

\begin{proposition}
\begin{eqnarray}\label{relation1}
s_{ij|k}=-b^m R_{kmij}+r_{ik|j}-r_{jk|i},
\end{eqnarray}
where $R_{kmij}$ is the fourth-order Riemann tensor determined by $R_{kmij}=\frac{1}{3}\left(\frac{\partial^2R_{mi}}{\partial y^j\partial y^k}-\frac{\partial^2R_{mj}}{\partial y^i\partial y^k}\right)$. As corollaries,
\begin{eqnarray}
\sIkho&=&\RoIkb-\RbIko-\rkohI+\rIohk,\label{relation7}\\
\sIohk&=&\RoIkb+\rIkho-\rkohI,\label{relation8}\\
\sIoho&=&-\RIb+\rIoho-\roohI,\label{relation9}\\
\sIohb&=&-\qIo+\qoI+\rIho-\rohI,\label{relation14}\\
\sIhk&=&-\RbIkb-\pIk-\tIk-\rIkhb+\rkhI,\label{relations10}\\
\skho&=&-\Rbokb-\pko-\tko-\rkohb+\rohk,\label{relation12}\\
\sohk&=&-\Rbokb-\pko-\tko-\rkohb+\rkho,\label{relation13}\\
\soho&=&-\Rbb-p_{00}-t_{00}-r_{00|b}+r_{0|0},\label{relation4}\\
\sIohi&=&\Rico+\rIiho-\rIohi,\label{relation2}\\
\sIhi&=&-\Ric-\pIi-\tIi-\rIihb+\rIhi.\label{relation6}
\end{eqnarray}
\end{proposition}
\begin{proof}
By Ricci identity, we have
\begin{eqnarray}
b_{i|j|k}-b_{i|k|j}&=&b^{m}R_{imjk},\nonumber\\
-b_{k|i|j}+b_{k|j|i}&=& -b^{m}R_{kmij},\nonumber\\
b_{j|k|i}-b_{j|i|k}&=&b^{m}R_{jmki}.\label{Ricciidentity}
\end{eqnarray}
On  the other hand, by the definition of $\rij$ we have
\begin{eqnarray}
b_{i|k|j}+b_{k|i|j}&=&2r_{ik|j},\nonumber\\
-b_{k|j|i}-b_{j|k|i}&=&-2r_{kj|i}.\label{rikj}
\end{eqnarray}
Adding all the equalities in (\ref{Ricciidentity}) and (\ref{rikj}) yields (\ref{relation1}). The argument above was given in \cite{db-robl-orso}.

(\ref{relation7})-(\ref{relation6}) can be obtain by contracting (\ref{relation1}) with related tensors and using some basic facts listed bolow,
\begin{eqnarray*}
\bi\rIoho&=&\roho-\poo-\qoo,\\
\bi\rIohb&=&\rohb-\po-\qqo,\\
\bi\rIhk&=&\rhk-\pk-\qk,\\
\bi\sIohk&=&\sohk-\qko+\tko,\\
\bi\sIkho&=&\skho-\qok+\tko,\\
\bK\sIkhi&=&-\sIhi-\qIi-\tIi,\\
\bi\sIhk&=&\qqk-\tk.
\end{eqnarray*}
Notice that in (\ref{relation6}) we use a fact $q^{i}{}_{i}=0$. It holds because
\begin{eqnarray*}
q^{i}{}_{i}=r_{ik}a^{lk}s_{lk}a^{ij}=-r_{ki}a^{ij}s_{jl}a^{lk}=-q_{kl}a^{lk}=-q^{l}{}_{l}.
\end{eqnarray*}
\end{proof}

\section{Randers norms and navigation deformations}
A Randers norm $F=\a+\b$ on a $n$-dimensional vector space $V^n$ can always be normalized as
\begin{eqnarray*}
F(y^1,y^2,\cdots,y^n)=\sqrt{(y^1)^2+(y^2)^2+\cdots+(y^n)^2}+by^n,
\end{eqnarray*}
where $b$ is a constant signifying the length of $\b=by^n$ with respect to the standard Euclid norm $\a=\sqrt{(y^1)^2+(y^2)^2+\cdots+(y^n)^2}$. Hence, the indicatrix of $F$, namely the set of unit vectors
\begin{eqnarray*}
\mathcal{I}^+_F:=\{y\in V^n~|~F(y)=1\},
\end{eqnarray*}
is a hypersurface in $V^n$ described by the equation
\begin{eqnarray}\label{indicatrix}
(y^1)^2+(y^2)^2+\cdots+(y^{n-1})^2+(1-b^2)(y^n)^2+2by^n=1,\quad 1-by^n\geq0.
\end{eqnarray}

It is known that $F$ is strongly convex if and only if $b<1$. In this case, $\mathcal{I}^+_F$ is elliptic. When $b=1$, (\ref{indicatrix}) represents a parabolic hypersurface, and there is only one direction $y$ satisfying $F(y)=0$. For the third case $b>1$,  either $\mathcal{I}^+_F$ or $\mathcal{I}^-_F:=\{y\in V^n~|~F(y)=-1\}$ is one sheet of two-sheet hyperboloid. Moreover, the set of null vectors, namely the vectors satisfying $F(y)=0$, is a semi-cone.

\begin{figure}[h]
\centering
\includegraphics[scale=0.7]{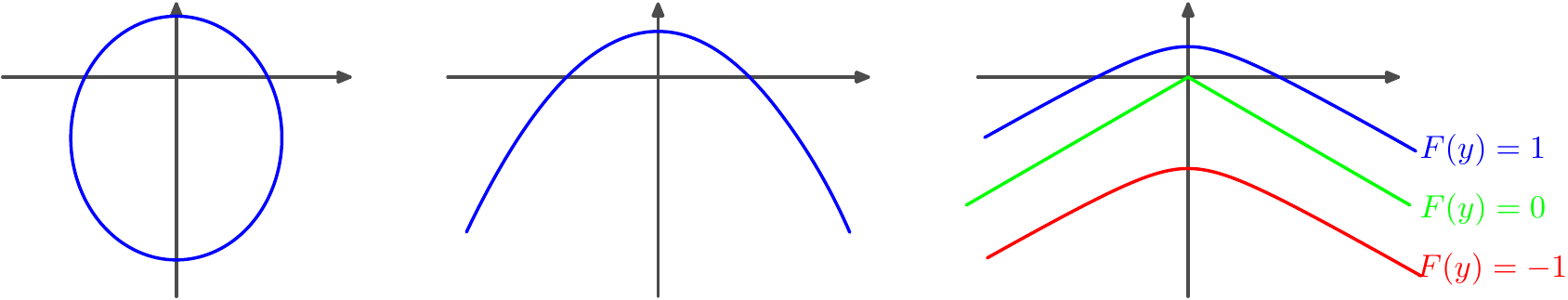}
\caption{Indicatrix of Randers norm when $b<1$, $b=1$ or $b>1$}
\end{figure}

In 1931, Zermelo considered the paths of shortest travel time problem on Euclid plane $\RR^2$, under the influence of a wind represented by a vector field on $\RR^2$\cite{zermelo}. It 2001, it is Shen who firstly realized that strongly convex Randers metrics are just the solutions of Zermelo's navigation problem on arbitrary Riemannian spaces\cite{db-robl-szm-zerm}. Navigation problem is also the key to reveal the properties of geodesics for Randers metrics or Kropina metrics\cite{Robles,Sabau-Shi}. Here we just state some basic facts about Shen's construction. See \cite{db-robl-szm-zerm} for details. See also \cite{huang} for more discussions on general navigation problem on an arbitrary Finsler space.

Consider a Euclid vector space $(V^n,\ba)$. $W$ is a vector with a restriction $\|W\|_{\ba}<1$. Then the following equation
\begin{eqnarray}\label{diaoiningannng}
\left\langle \f{y}{F(y)}-W,\f{y}{F(y)}-W\right\rangle_{\ba}=1,
\end{eqnarray}
determines a Minkowski norm on $V^n$. Notice that (\ref{diaoiningannng}) is a quadratic equation of $F$
\begin{eqnarray*}
(1-\bar b^2)F^2+2\bb F-\ba^2=0,
\end{eqnarray*}
where $\bb:=\langle y,W\rangle_{\ba}$ is the dual linear functional of $W$, $\bar b:=\|\bb\|_{\ba}=\|W\|_{\ba}$. As a result,
\begin{eqnarray*}
F(y)=\f{\sqrt{(1-\bar b^2)\ba^2+\bb^2}}{1-\bar b^2}-\f{\bb}{1-\bar b^2},
\end{eqnarray*}
which can be rewrote as $F=\a+\b$ where $\a=\frac{\sqrt{(1-\bar b^2)\ba^2+\bb^2}}{1-\bar b^2}$ and $\b=-\frac{\bb}{1-\bar b^2}$. Notice that $b=\bar b<1$. so $F$ is a strongly convex Randers norm. Conversely, $\ba$ and $\bb$ can be determined by $\a$ and $\b$ as below,
\begin{eqnarray}\label{donaienggainiga}
\ba=\sqrt{(1-b^2)(\a^2-\b^2)},\qquad \bb=-(1-b^2)\bb.
\end{eqnarray}
We will call (\ref{donaienggainiga}) the \emph{navigation deformations} for strongly convex Randers norm.

In fact, for the case $b>1$, $F$ can also be regarded as a solution of navigation problem. But we need to replace the underlying Euclid space as a Lorentz space. Let $\ba$ be a Lorentz norm on $V^n$ with Lorentz signature $(-,\cdots,-,+)$ and $W$ be a timelike vector with a restriction $\|W\|_{\ba}>1$. Then Equation (\ref{diaoiningannng}) determines a Minkowski norm
\begin{eqnarray*}
F(y)=\f{\sqrt{(1-\bar b^2)\ba^2+\bb^2}}{\bar b^2-1}+\f{\bb}{\bar b^2-1},
\end{eqnarray*}
in which $\bb:=\langle W,y\rangle_{\ba}$, $\bar b:=\|\bb\|_{\ba}=\|W\|_{\ba}>1$. It can be rewrote as $F=\a+\b$ too, where $\a=\frac{\sqrt{(1-\bar b^2)\ba^2+\bb^2}}{\bar b^2-1}$ and $\b=\frac{\bb}{\bar b^2-1}$. It is easy to verify that  $\a$ is positive definite on $V^n$, and $b:=\|\b\|_{\a}=\|\bb\|_{\ba}>1$. Conversely, $\ba$ and $\bb$ can be determined by $\a$ and $\b$ also as (\ref{donaienggainiga}),

Roughly speaking, shifting a Euclid norm based on
\begin{eqnarray*}
\mathcal{I}^+_F=\mathcal{I}^+_{\ba}+\{W\}
\end{eqnarray*}
yields a strongly convex Randers norm, and shifting a Lorentz norm based on
\begin{eqnarray*}
\mathcal{I}^+_F\subset\mathcal{I}^+_{\ba}+\{W\}
\end{eqnarray*}
yields a Rander norm $F=\a+\b$ with $b>1$. Notice that for the later case, only one sheet of the two-sheet hyperboloid after shifting is reserved.

\begin{figure}[h]
\centering
\includegraphics[scale=0.85]{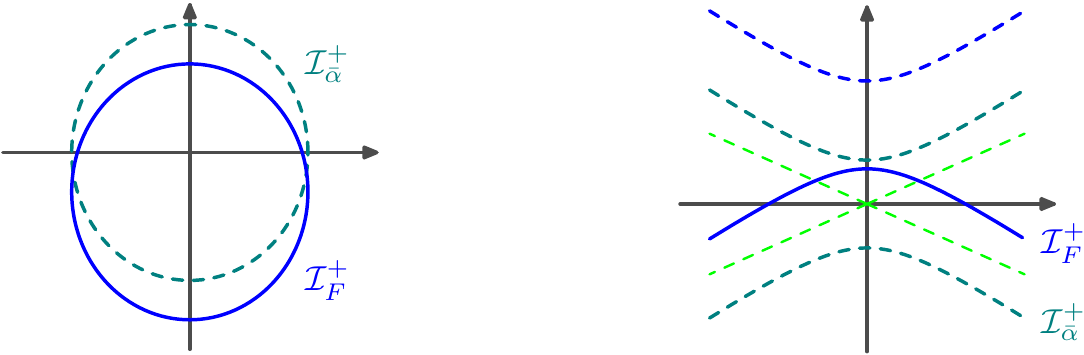}
\caption{Randers norm when $b<1$ or $b>1$ under navigation deformations}
\end{figure}

The restriction $\|W\|_{\ba}>1$ should be emphasized.

If $W$ is timelike with $\|W\|_{\ba}<1$, then one can obtain a Randers type pseudo norm as $F=\frac{\sqrt{(1-\bar b^2)\ba^2+\bb^2}}{1-\bar b^2}-\frac{\bb}{1-\bar b^2}$ by (\ref{diaoiningannng}) since $\a=\frac{\sqrt{(1-\bar b^2)\ba^2+\bb^2}}{1-\bar b^2}$ is a Lorentz norm with signature $(-,\cdots,-,+)$. Similarly, one can obtain another pseudo Randers norm beginning with a Lorentz norm $\ba$ and its spacelike vector $W$ with $-1<\|W\|^2_{\ba}<0$ or $\|W\|^2_{\ba}>-1$ combining with another equation
\begin{eqnarray}\label{donwieniganinag}
\left\langle -\f{y}{F(y)}-W,-\f{y}{F(y)}-W\right\rangle_{\ba}=-1,
\end{eqnarray}
We claim, without any argument, that such metrics are of simple curvature structure similar to Randers metrics with $b\neq1$ (see Theorem \ref{uabubdagnianengga}). The key reason is that all of them can be regarded as solutions of some particular navigation problem, hence navigation deformations works.

Recently, M.A. Javaloyes and M. S¨¢nchez studied a generalized navigation problem on Riemann space $(M,\ba)$ without the restriction of the wind field $W$ satisfying $\|W\|_{\ba}<1$\cite{JS}, including three cases: $\|W\|_{\ba}<1$, $\|W\|_{\ba}\equiv1$ and $\|W\|_{\ba}>1$. In particular, the case when $\|W\|_{\ba}\equiv1$ leads to the famous Kropina metric $F=\frac{\b^2}{\a}$. Obviously, their argument is different form us. However, it is worth to be remarked that when $\ba$ is a Lorentz metric and $W$ is a timelike (resp. spacelike) vector field with $\|W\|_{\ba}\equiv1$ (resp. $\|W\|^2_{\ba}\equiv-1$), then one can obtain Kropina-type metrics according to equation (\ref{diaoiningannng})
(resp. (\ref{donwieniganinag})). Similarly, we claim, without any argument, that both of them are of simple curvature structure similar to classical Kropina metrics\cite{zhang-shen}.

\section{A characterization for Randers-Einstein metrics}
In this section, we will characterize a Randers metric of constant Ricci curvature without the restriction of strong convexity. The whole discussion is similar to that in \cite{db-robl-onri}, but more concise. And, most importantly, one can see clearly what happens when $b\equiv1$.

First, it is known that the Ricci curvature of a Randers metric $F=\a+\b$ is given by
\begin{eqnarray*}
\RicooF&=&\Ricoo+\f{3}{4}(n-1)(\a+\b)^{-2}(\roo-2\a\so)^2+2(n-1)\a(\a+\b)^{-1}\qoo-2\too\\
&&-2(n-1)\a^2(\a+\b)^{-1}\to-\a^2\tIi+2\a\sIohi-\f{1}{2}(n-1)(\a+\b)^{-1}\rooho+(n-1)\a(\a+\b)^{-1}\soho,
\end{eqnarray*}
where $\Ricoo$ is the Ricci curvature of $\a$\cite{cxy-szm-fgaa}.

Assume $F$ is an Einstein metric with Ricci constant $K$, i.e.,
\begin{eqnarray*}
\RicooF=(n-1)K(\a+\b)^2.
\end{eqnarray*}
It can be rewrote as
\begin{eqnarray*}
\Rat+\a\Irrat=0,
\end{eqnarray*}
where
\begin{eqnarray*}
\Rat&=&4(\a^2+\b^2)\Ricoo-4(n-1)K(\a^4+6\a^2\b^2+\b^4)+3(n-1)\roo^2+12(n-1)\a^2\so^2+8(n-1)\a^2\qoo\\
&&-8(\a^2+\b^2)\too-8(n-1)\a^2\b\to-4\a^2(\a^2+\b^2)\tIi-2(n-1)\b\rooho+4(n-1)\a^2\soho+16\a^2\b\sIohi,\\
\Irrat&=&8\b\Ricoo-16(n-1)K(\a^2+\b^2)\b-12(n-1)\roo\so+8(n-1)\b\qoo-16\b\too-8(n-1)\a^2\to\\
&&-8\a^2\b\tIi-2(n-1)\rooho+4(n-1)\b\soho+8(\a^2+\b^2)\sIohi.
\end{eqnarray*}
Since $\a$ is irrational on $y$ but both of $\Rat$ and $\Irrat$ are rational, we have
\begin{eqnarray*}
\Rat=0,\qquad\Irrat=0.
\end{eqnarray*}

By
\begin{eqnarray*}
\Rat-\b\Irrat&=&4(\a^2-\b^2)\big\{\Ricoo-(n-1)K(\a^2+3\b^2)+3(n-1)\so^2+2(n-1)\qoo\\
&&-2\too-\a^2\tIi+(n-1)\soho+2\b\sIohi\big\}+3(n-1)(\roo+2\b\so)^2\\
&=&0,
\end{eqnarray*}
The polynomial $\a^2-\b^2$ must divide $\roo+2\b\so$ exactly. So exist a scalar function $c(x)$ such that
\begin{eqnarray}
\roo=c(x)(\a^2-\b^2)-2\b\so.\label{yawbbwjbabgg}
\end{eqnarray}
By (\ref{yawbbwjbabgg}) we can calculate the related terms such as $\rooho$, $\roho$, $\poo$ and $\qoo$, etc. As a result, $\Rat$ and $\Irrat$ become
\begin{eqnarray*}
\Rat&=&4(\a^2+\b^2)\Ricoo-4(n-1)K(\a^2+6\a^2\b^2+\b^4)+(n-1)\big\{(\a^2-\b^2)(3\a^2+\b^2)c^2\\
&&-16\a^2\b c\so-2(\a^2-\b^2)\b\co\big\}+4(n-1)(\a^2+\b^2)\so^2-8(\a^2+\b^2)\too-16(n-1)\a^2\b\to\\
&&-4\a^2(\a^2+\b^2)\tIi+4(n-1)(\a^2+\b^2)\soho+16\a^2\b\sIohi,\\
\Irrat&=&8\b\Ricoo+2(n-1)\big\{2(\a^2-\b^2)\b c^2-4(\a^2+\b^2)c\so-(\a^2-\b^2)\co\big\}-16(n-1)K(\a^2+\b^2)\b\\
&&+8(n-1)\b\so^2-16\b\too-8(n-1)(\a^2+\b^2)\to-8\a^2\b\tIi+8(n-1)\b\soho+8(\a^2+\b^2)\sIohi,
\end{eqnarray*}
wherh $\co:=c_{|i}y^i$ and $\cb:=c_{|i}b^i$.

By
\begin{eqnarray*}
2\b\Rat-(\a^2+\b^2)\Irrat&=&2(\a^2-\b^2)^2\left\{(n-1)\big[(4K+c^2)\b+\co+4c\so+4\to\big]-4\sIohi\right\}=0
\end{eqnarray*}
we have
\begin{eqnarray}
\sIohi=\f{1}{4}(n-1)\big\{(4K+c^2)\b+\co+4c\so+4\to\big\},\label{tevbagbenga}
\end{eqnarray}
So by $\Rat=0$ (it is already equivalent to $\Irrat=0$) we have
\begin{eqnarray*}
\Ricoo&=&\f{1}{4}\big\{[(n-1)(4K-3c^2)+4\tIi]\a^2+(n-1)(4K+c^2)\b^2\\
&&-2(n-1)\b\co-4(n-1)\so^2+8\too-4(n-1)\soho\big\}.
\end{eqnarray*}
As a result,
\begin{eqnarray*}
\Rico&=&\f{1}{4}\big\{[4(n-1)K(1+b^2)-(n-1)(3-b^2)c^2-(n-1)\cb+4\tIi+2(n-1)t]\b-(n-1)b^2\co\\
&&+2(n-1)c\so+2(n+3)\to-2(n-1)\sohb\big\},\\
\Ric&=&\f{1}{4}\big\{4(n-1)Kb^2(1+b^2)-(n-1)b^2(3-b^2)c^2-2(n-1)b^2\cb+4b^2\tIi+4[(n+1)+(n-1)b^2]t\big\},
\end{eqnarray*}
where $\Rico:=Ric_{ij}b^iy^j$ and $\Ric:=Ric_{ij}b^ib^j$ according to our rules in Section 2. Moreover, differentiating  (\ref{tevbagbenga}) with respect to $y^k$ and contracting with $b^k$ yields
\begin{eqnarray}
\sIhi=-\f{1}{4}\big\{(n-1)(4K+c^2)b^2+(n-1)\cb+4\tIi+4(n-1)t\big\}.\label{aidninagadgag}
\end{eqnarray}

Finally, we use the priori formulae in Section 2 to obtain more latent facts. Actually, there are only three formulae can be used here, namely (\ref{relation5}), (\ref{relation2}) and (\ref{relation6}). They read
\begin{eqnarray*}
&\displaystyle(1-b^2)\big\{\sohb+(\cb+t)\b-b^2\co-c\so+\to\big\}=0,&\\
&\displaystyle\f{1}{4}\big\{-2(n-3)[\sohb+(\cb+t)\b-b^2\co-c\so+\to]+3(n-1)(1-b^2)\co\big\}=0
\end{eqnarray*}
and
\begin{eqnarray*}
\frac{3}{4}(n-1)(1-b^2)\cb=0
\end{eqnarray*}
respectively. Hence, when $b\neq1$, the above equalities are equivalent to $\co=0$ and
\begin{eqnarray}
\sohb=c\so-\to-\b t.\label{dinaigneinagad}
\end{eqnarray}
But when $b\equiv1$, they indicate
\begin{eqnarray*}
(n-3)(\co-\b\cb+c\so-\to-\b t-\sohb)=0
\end{eqnarray*}
merely.

Summarizing the above discussions, we have the following conclusion.
\begin{theorem}\label{miandinginaindnga1}
Randers metric $F=\a+\b$ is an Einstein metric with Ricci constant $K$ if and only if
\begin{itemize}
\item for the case $b\neq1$ ,
\begin{eqnarray}
\Ricoo&=&\left\{(n-1)\left(K-\f{3}{4}c^2\right)+\tIi\right\}\a^2+(n-1)\left(K+\f{1}{4}c^2\right)\b^2\nonumber\\
&&-(n-1)\so^2+2\too-(n-1)\soho,\label{ebuabgabdgb1}\\
\roo&=&c(\a^2-\b^2)-2\b\so,\label{ububanajgnabdg1}
\end{eqnarray}
where $c$ is a constant;
\item for the case $b\equiv1$,
\begin{eqnarray}
\Ricoo&=&\left\{(n-1)\left(K-\f{3}{4}c^2\right)+\tIi\right\}\a^2+(n-1)\left(K+\f{1}{4}c^2\right)\b^2\nonumber\\
&&-\f{1}{2}(n-1)\b\co-(n-1)\so^2+2\too-(n-1)\soho,\label{ebuabgabdgb2}\\
\roo&=&c(\a^2-\b^2)-2\b\so,\label{ububanajgnabdg2}
\end{eqnarray}
where $c=c(x)$ is a scalar function satisfying an additional condition
\begin{eqnarray}
(n-3)(\co-\b\cb+c\so-\to-\b t-\sohb)=0.\label{dnaiegbabgadn}
\end{eqnarray}
\end{itemize}
\end{theorem}
\begin{proof}
When $b\neq1$, combining with the discussions above, it is clearly that when (\ref{ebuabgabdgb1}), (\ref{ububanajgnabdg1}) and (\ref{tevbagbenga})-(\ref{dinaigneinagad}) hold, then $F=\a+\b$ has constant Ricci curvature $K$. So these five equalities become the necessary and sufficient conditions for $F$ to be an Einstein metric automatically. However, the last three equalities can be rebuilt by using the first two equalities combining with the priori formulae (\ref{relation2}), (\ref{relation6}) and (\ref{relation5}). Hence, the necessary and sufficient conditions are reduced as (\ref{ebuabgabdgb1}) and (\ref{ububanajgnabdg1}).

Similarly, when $b\equiv1$, (\ref{ebuabgabdgb2}), (\ref{ububanajgnabdg2}), (\ref{dnaiegbabgadn}), (\ref{tevbagbenga}) and (\ref{aidninagadgag}) are the necessary and sufficient conditions for $F$ to be an Einstein metric. Among these five equalities, the last two can be rebuilt by using the first three equalities combining with the priori formulae (\ref{relation2}) and (\ref{relation6}). But be attention that the equality (\ref{dnaiegbabgadn}) isn't able to be rebuilt by any priori formula listed in Section 3. Hence, the necessary and sufficient conditions are reduced as (\ref{ebuabgabdgb2})-(\ref{dnaiegbabgadn}).
\end{proof}

\section{A characterization for Randers metrics of constant flag curvature}
First, it is known that the Riemann tensor of a Randers metric  $F=\a+\b$ is given by
\begin{eqnarray*}
\RIkF&=&\RIk+\f{1}{4}(\a+\b)^{-2}\big\{3(\roo-2\a\so)^2-8\a^2(\a+\b)\to+8\a(\a+\b)\qoo-2(\a+\b)\rooho\\
&&+4\a(\a+\b)\soho\big\}\dIk-\f{1}{4}\a^{-1}(\a+\b)^{-3}\big\{3(\roo-2\a\so)^2+4(\a+\b)(3\a+\b)\qoo\\
&&-4\a(\a^2-\b^2)\to-4(\a+\b)\b\soho-2(\a+\b)\rooho\big\}\yI\yk-\f{1}{4}(\a+\b)^{-3}\big\{3(\roo-2\a\so)^2\\
&&+8\a(\a+\b)\qoo-8\a^2(\a+\b)\to-2(\a+\b)\rooho+4\a(\a+\b)\soho\big\}\yI\bk-3(\a+\b)^{-1}\so\yI\sko\\
&&+\a(\a+\b)^{-1}\yI(2\qko-\qok)+\a^2(\a+\b)^{-1}\yI\tk+(\a+\b)^{-1}\yI(\roohk-\rkoho)\\
&&-\a(\a+\b)^{-1}\yI(2\sohk-\skho)+\tIo\yk-\a^{-1}\sIoho\yk+3\sIo\sko-\a^2\tIk+\a(2\sIohk-\sIkho).
\end{eqnarray*}
where $\RIk$ is the Riemann curvature of $\a$\cite{cxy-szm-fgaa}.

Assume $F$ has constant flag curvature $K$, i.e.,
\begin{eqnarray*}
\RIkF=K\left\{(\a+\b)^2\dIk-(\a+\b)\yI\left(\f{\yk}{\a}+\bk\right)\right\}.
\end{eqnarray*}
Recall that $y_k:=a_{kl}y^l$.

Plugging (\ref{yawbbwjbabgg}) (it holds since $F$ has constant Ricci curvatur) and all the related terms (such as $\rooho$, $\roho$, $\poo$ and $\qoo$, etc.) into the above equality, it can be rewrote as
\begin{eqnarray*}
(\a+\b)^3(\rat+\a\irrat)=0,
\end{eqnarray*}
where
\begin{eqnarray*}
\rat&=&-4\bigg\{2\left[\left(K+\f{1}{4}c^2\right)\b+\f{1}{4}\co+c\so+\to\right]\a^2\dIk-\left[\left(K+\f{1}{4}c^2\right)\b-\f{1}{2}\co+c\so+\to\right]\yI\yk\\
&&-\left(K+\f{1}{4}c^2\right)\a^2\yI\bk-\a^2\yI\ck-\a^2c\yI\sk-\a^2\yI\tk+\sIoho\yk+\a^2(\sIkho-2\sIohk)\bigg\},\\
\irrat&=&-4\bigg\{\left[\left(K-\f{3}{4}c^2\right)\a^2+(K+\f{1}{4}c^2)\b^2-\f{1}{2}\co\b-\so^2-\soho\right]\dIk-\left(K-\f{3}{4}c^2\right)\yI\yk\\
&&-\left[\left(K+\f{1}{4}c^2\right)\b+\f{1}{2}\co\right]\yI\bk+\b\yI\ck-3c\yI\sko+\so\yI\sk-\yI\skho+2\yI\sohk-\tIo\yk\\
&&-3\sIo\sko+\a^2\tIk-\RIk\bigg\}.
\end{eqnarray*}
Since $\a$ is irrational on $y$ but both of $\rat$ and $\irrat$ are rational, we have
\begin{eqnarray*}
\rat=0,\qquad\irrat=0.
\end{eqnarray*}

Solving $\irrat=0$ yields
\begin{eqnarray}
\RIk&=&\left\{\left(K-\f{3}{4}c^2\right)\a^2+\left(K+\f{1}{4}c^2\right)\b^2-\f{1}{2}\co\b-\so^2-\soho\right\}\dIk\nonumber\\
&&-\left(K-\f{3}{4}c^2\right)\yI\yk-\left\{\left(K+\f{1}{4}c^2\right)\b+\f{1}{2}\co\right\}\yI\bk+\b\yI\ck\nonumber\\
&&-3c\yI\sko+\so\yI\sk-\yI\skho+2\yI\sohk-\tIo\yk-3\sIo\sko+\a^2\tIk,\label{einqieabbbg}
\end{eqnarray}
hence
\begin{eqnarray*}
\RIb&=&-\bigg\{\left(K-\f{3}{4}c^2\right)\b+\left(K+\f{1}{4}c^2\right)b^2\b+\f{1}{2}b^2\co-\b\cb+2c\so-\to-\b t-2\sohb\bigg\}\yI\\
&&+\bigg\{\left(K-\f{3}{4}c^2\right)\a^2+\left(K+\f{1}{4}c^2\right)\b^2-\f{1}{2}\b\co-\so^2-\soho\bigg\}\bI-3\so\sIo-\b\tIo+\a^2\tI,\\
\Rbb&=&\bigg\{\left(K-\f{3}{4}c^2\right)b^2+t\bigg\}\a^2-\left(K-\f{3}{4}c^2-\cb-t\right)\b^2-b^2\b\co-(3+b^2)\so^2-2\b c\so-b^2\soho+2\b\sohb,
\end{eqnarray*}
where $\RIb:=\RIk b^k$ and $\Rbb:=\RIb b_i$.

The priori formulae (\ref{relation5}) and (\ref{relation4}) read
\begin{eqnarray}
(1-b^2)(b^2\co-\b\cb+c\so-\to-\b t-\sohb)=0\label{fatsvfvag1}
\end{eqnarray}
and
\begin{eqnarray}
-\left(K+\f{1}{4}c^2\right)(b^2\a^2-\b^2)+\b\co-\a^2\cb-\so^2-\too-\a^2t-\soho=0\label{fatsvfvag2}
\end{eqnarray}
respectively. By (\ref{fatsvfvag2}),
\begin{eqnarray*}
\soho=-\left(K+\f{1}{4}c^2\right)(b^2\a^2-\b^2)+\b\co-\a^2\cb-\so^2-\too-\a^2t.
\end{eqnarray*}
Moreover, by the above equality we have
\begin{eqnarray*}
\sohb=b^2\co-\b\cb+c\so-\to-\b t.
\end{eqnarray*}
Notice that such equality can't be obtained by (\ref{fatsvfvag1}) directly since $b$ is possible equal to $1$.

Now,
\begin{eqnarray*}
\rat b_ib^k=6(b^2\a^2-\b^2)\co=0.
\end{eqnarray*}
Hence, $c(x)$ must be a constant. In this case,
\begin{eqnarray*}
\rat b^k&=&4\bigg\{\left[\left(K+\f{1}{4}c^2\right)(b^2\a^2+\b^2)+ c\b \so+\b\to+\a^2t\right]\yI-2\a^2\left[\left(K+\f{1}{4}c^2\right)\b+c\so+\to\right]\bI\\
&&+c\a^2\sIo+\a^2(c\b+\so)\sI+\a^2\tIo+\a^2\b\tI-\b\sIoho+2\a^2\sIohb+\a^2\sIho\bigg\},\\
&=&0.
\end{eqnarray*}
By (\ref{einqieabbbg}) we can get the fourth-order Riemann curvature tensor $R_j{}^i{}_{kl}$. Combining with the priori formulea (\ref{relation9}), (\ref{relation14}) and (\ref{relation13}), the above equality reads
\begin{eqnarray*}
\rat b^k&=&-4\big\{(2b^4+2b^2-1)\a^2-(2b^2+1)\b^2\big\}\bigg\{\left(K+\f{1}{4}c^2\right)(b^2\yI-\b\bI)+t\yI+c\sIo+\so\sI+\tIo+\sIho\bigg\}=0.
\end{eqnarray*}
Hence,
\begin{eqnarray}
\sIhk=-\left(K+\f{1}{4}c^2\right)(b^2\dIk-\bI\bk)-t\dIk-c\sIk-\sI\sk-\tIk.\label{dinwuebuabgaidng}
\end{eqnarray}

Finally, $\rat$ is given by
\begin{eqnarray*}
\rat&=&-4\a^2\bigg\{2\left[\left(K+\f{1}{4}c^2\right)\b+c\so+t_0\right]\dIk-\left(K+\f{1}{4}c^2\right)(\yI\bk+\bI\yk)\\
&&-c(\yI\sk+\sI\yk)-(\yI\tk+\tI\yk)+\sIkho-2\sIohk\bigg\}.
\end{eqnarray*}
$\rat=0$ holds automatically after plugging the priori formulae (\ref{relation7}) and (\ref{relation8}) into it.

Summarizing the above discussions, we have the following conclusion.
\begin{theorem}\label{miandinginaindnga2}
Randers metric $F=\a+\b$ is of constant flag curvature $K$ if and only if exists a constant $c$ such that
\begin{eqnarray}
\RIk&=&\bigg\{\left(K-\f{3}{4}c^2\right)+\left(K+\f{1}{4}c^2\right)b^2+t\bigg\}(\a^2\dIk-\yI\yk)\nonumber\\
&&-3\sIo\sko-\yI\tko-\tIo\yk+\too\dIk+\a^2\tIk,\label{indiaaudgngad}\\
\rij&=&c(\aij-\bi\bj)-\bi\sj-\bj\si\label{ainiaudbuabgbag},
\end{eqnarray}
when $b\equiv1$, $\b$ should satisfy an additional condition
\begin{eqnarray}
s_{i|j}-s_{j|i}=-2cs_{ij}.\label{addsij}
\end{eqnarray}
\end{theorem}
\begin{proof}
Necessity: Plugging (\ref{dinwuebuabgaidng}) and $c_k=0$ into (\ref{einqieabbbg}) yields (\ref{indiaaudgngad}). (\ref{ainiaudbuabgbag}) holds due to (\ref{yawbbwjbabgg}). (\ref{addsij}) can be obtained by (\ref{dinwuebuabgaidng}).

Sufficiency: Assume (\ref{indiaaudgngad}) and (\ref{ainiaudbuabgbag}) hold, in which $c$ is a constant. Firstly, the priori formula (\ref{relation5}) reads
\begin{eqnarray*}
(1-b^2)(\sohb-c\so+\to+\b t).
\end{eqnarray*}

Hence, when $b\neq1$, $\sohb=c\so-\to-\b t$. In this case, combining with the priori formula (\ref{relation13}), (\ref{relation12}) reads
\begin{eqnarray*}
(1-b^4)\left\{\left(K+\f{1}{4}c^2\right)(b^2\yI-\b\bI)+t\yI+c\sIo+\so\sI+\tIo+\sIho\right\}=0,
\end{eqnarray*}
which indicates (\ref{dinwuebuabgaidng}). Finally, by priori formulae (\ref{relation7}) and (\ref{relation8}) we know that $\rat=\irrat=0$, so $F=\a+\b$ has constant flag curvature $K$.

However, when $b\equiv1$, (\ref{dinwuebuabgaidng}) can't be rebuilt by (\ref{indiaaudgngad}), (\ref{ainiaudbuabgbag}) and the priori formulae (\ref{relation5})-(\ref{relation6}). That is to say, (\ref{indiaaudgngad}) and (\ref{ainiaudbuabgbag}) will not lead to $\rat=\irrat=0$, unless (\ref{dinwuebuabgaidng}) holds too. So (\ref{dinwuebuabgaidng}), (\ref{indiaaudgngad}) and  (\ref{ainiaudbuabgbag}) are sufficient to make $F=\a+\b$ be of constant flag curvature. Finally, by (\ref{addsij}) we have $\skhI=\sIhk+2c\sIk$ and $s_{k|b}=c\sk-\tk-t\bk$, so (\ref{dinwuebuabgaidng}) can be built by (\ref{addsij}) combining with the priori formula (\ref{relations10}). Hence, (\ref{dinwuebuabgaidng}) can be replaced with (\ref{addsij}).
\end{proof}

\begin{remark}
In the original version of Bao-Robles's characterization for strongly convex Randers metrics, there are three equations including the Curvature equation (\ref{indiaaudgngad}), the Basic equation (\ref{ainiaudbuabgbag}) and an extra equation termed CC(23) in \cite{db-robl-orso}. Later, they realized that the extra equation is derivable from the Basic and Curvature equations with the priori formula (\ref{relation1}). However, the additional condition (\ref{addsij}) for the case $b\equiv1$ is not derivable from the Basic and Curvature equations, as we have shown in the proof above.
\end{remark}

\section{Classification for Randers metrics with constant flag or Ricci curvature when $b\neq1$}
Obviously, for a Randers metric with constant flag or Ricci curvature, the geometrical properties of the original data $\a$ and $\b$ are daunting. Based on Zermelo's navigation problem, Bao-Robles-Shen found that the navigation deformations (\ref{donaienggainiga}) can make clear the underlying geometry of $\a$ and $\b$ for strongly convex Randers metrics. Their well-known result is concluded in the first case of Theorem \ref{uabubdagnianengga}. It should be attention that the notation `$\bbij$' always means the covariant derivative of $1$-form $\bb$ with respect to the corresponding Riemann metric $\ba$.

Actually, the case $b>1$ is similar to the strongly convex case. The computations in \cite{db-robl-onri,db-robl-orso,db-robl-szm-zerm} remain valid for this case. So we give our result directly and the proof is omitted.

\begin{theorem}[Classification]\label{uabubdagnianengga}
Let $F=\a+\b$ be a Randers metric on a $n$-dimensional manifold $M$ with $b\neq1$ anywhere. Then $F$ is of constant flag (resp. Ricci) curvature $K$ if and only if
\begin{enumerate}[(a)]
\item when $b<1$,
\begin{eqnarray}\label{Nd1}
\ba^2=(1-b^2)(\a^2-\b^2)
\end{eqnarray}
is a Riemann metric of constant sectional (resp. Ricci) curvature $\mu$,
\begin{eqnarray}\label{Nd2}
\bb=-(1-b^2)\b
\end{eqnarray}
is homothetic to $\ba$ with homothetic factor $-c$, namely $\bbij=-c\baij$. In this case,
\begin{eqnarray}\label{dnaienignadng}
F=\f{\sqrt{(1-\bar b^2)\ba^2+\bb^2}}{1-\bar b^2}-\f{\bb}{1-\bar b^2}.
\end{eqnarray}
\item when $b>1$, (\ref{Nd1}) is a Lorentz metric with signature $(-,\cdots,-,+)$ of constant sectional (resp. Ricci) curvature $\mu$, (\ref{Nd2}) is timelike and homothetic to $\ba$ with homothetic factor $-c$. In this case,
\begin{eqnarray}\label{diaiengnaingiang}
F=\f{\sqrt{(1-\bar b^2)\ba^2+\bb^2}}{\bar b^2-1}+\f{\bb}{\bar b^2-1}.
\end{eqnarray}
\end{enumerate}
For both cases, $\bar b=b$, $K=\mu-\frac{1}{4}c^2$.
\end{theorem}

All the Riemann space forms and their homothetic cotangent vector fields with homothetic factor $c$ (when the curvature $\mu\neq0$, $c$ must vanish) can be determined completely as
\begin{eqnarray}
\ba_\mu&=&\f{\sqrt{(1+\mu\xx)\yy-\mu\xy^2}}{1+\mu\xx},\label{uwubeubgadgad}
\end{eqnarray}
and
\begin{eqnarray*}
\bb&=&\left\{ \begin{aligned}
&c\xy+x^TQy+\langle a,y\rangle,\quad&\mu=0,\\
&x^TQy+\langle a,y\rangle,&\mu\neq0
\end{aligned} \right.\nonumber
\end{eqnarray*}
in some suitable local coordinate system, where $a$ is a constant victor and $Q$ is a constant antisymmetric matrix.

Hence, all the strongly convex Randers metrics can be completely determined locally. For instance, if $\ba=|y|$ is the standard Euclid metric and $\bb=\xy$, then the corresponding Randers metric (\ref{dnaienignadng})
\begin{eqnarray*}
F=\f{\sqrt{(1-\xx)\yy+\xy^2}}{1-\xx}-\f{\xy}{1-\xx}
\end{eqnarray*}
is the famous Funk's metric on open unit ball $\mathbb{B}^n(1)$, with constant flag curvature $K=-\frac{1}{4}$ and all the line segments as its geodesics. See \cite{db-robl-szm-zerm} for more discussions on strongly convex Randers metrics.

\section{Exact solutions of constant flag curvature with $b>1$}
Typical $4$-dimensional Lorentz metrics with constant sectional curvature include the Minkowski metric
\begin{eqnarray*}
\ud s^2=\ud t^2-\ud x^2-\ud y^2-\ud z^2
\end{eqnarray*}
with vanishing curvature $\mu=0$, the de Sitter metric
\begin{eqnarray*}
\ud s^2=\f{\ud t^2-\ud x^2-\ud y^2-\ud z^2}{\left\{1+\frac{1}{4}(t^2-x^2-y^2-z^2)\right\}^2}
\end{eqnarray*}
with positive curvature $\mu=1$, and the anti de Sitter metric
\begin{eqnarray*}
\ud s^2=\f{\ud t^2-\ud x^2-\ud y^2-\ud z^2}{\left\{1-\frac{1}{4}(t^2-x^2-y^2-z^2)\right\}^2}
\end{eqnarray*}
with nagetive curvature $\mu=-1$. The above metrics can also be expressed as
\begin{eqnarray*}
\ud s^2=(1-\mu r^2)\,\ud t^2-\f{\ud r^2}{1-\mu r^2}-r^2(\ud\theta^2+\sin^2\theta\,\ud\varphi^2).
\end{eqnarray*}

Actually, any $n$-dimensional Lorentz metric with constant sectional curvature $\mu$ must locally isometry to
\begin{eqnarray*}
\ba_\mu=\f{\sqrt{\langle y,y\rangle_L}}{1+\frac{\mu}{4}\langle x,x\rangle_{\mathrm L}},
\end{eqnarray*}
where
\begin{eqnarray*}
\xy_{\mathrm L}:=-x^1y^1\cdots-x^{n-1}y^{n-1}+x^ny^n
\end{eqnarray*}
is the standard Lorentz inner product. These metrics can also be expressed in projective coordinate systems as
\begin{eqnarray}\label{Lorentz-spaceform}
\ba_\mu=\f{\sqrt{(1+\mu|x|^2_{\mathrm L})|y|^2_{\mathrm L}-\mu\xy_{\mathrm L}^2}}{1+\mu|x|^2_{\mathrm L}},
\end{eqnarray}
just like the Riemann space forms (\ref{uwubeubgadgad}).

By the similar argument in Section 4 of \cite{db-robl-szm-zerm}, we can determine all the homothetic cotangent vector fields $\bb$ of (\ref{Lorentz-spaceform}) with homothetic factor $c$ completely:
\begin{eqnarray}\label{aoneingaingianng}
\bb=\left\{ \begin{aligned}
&c\xy_{\mathrm L}+x^TQy+\langle a,y\rangle_{\mathrm L},\quad&\mu=0,\\
&x^TQy+\langle a,y\rangle_{\mathrm L},&\mu\neq0,
\end{aligned} \right.
\end{eqnarray}
where $a$ is a constant victor and $Q$ is a constant antisymmetric matrix.

Hence, one can obtain many exact Randers metrics of constant flag curvature with $b>1$. Here we just write down a typical one: Let $\ba=\sqrt{\langle y,y\rangle_{\mathrm L}}$ be a flat Minkowski metric on $\RR^n$ and $\bb=\xy_L$ be its homothetic cotangent vector field with homethetic factor $c=1$. By Theorem \ref{uabubdagnianengga}, the following non-regular Randers metric
\begin{eqnarray*}
F=\f{\sqrt{(1-\xx_{\mathrm L})\yy_{\mathrm L}+\xy_{\mathrm L}^2}}{\xx_{\mathrm L}-1}+\f{\xy_{\mathrm L}}{\xx_{\mathrm L}-1}
\end{eqnarray*}
has constant flag curvature $K=-\frac{1}{4}$ on domain $|x|_{\mathrm L}>1$.

\section{Exact solutions of constant Ricci curvature with $b>1$}
In this section we will just focus on $4$-dimensional spaces. In $4$-dimension spacetimes, General Relativity provides many interesting models.

Einstein's field equations (EFE) are a highly nonlinear system of PDEs
\begin{eqnarray*}
R_{ij}-\f{1}{2}Rg_{ij}+\Lambda g_{ij}=\f{8\pi G}{c^4} T_{ij},
\end{eqnarray*}
where $g_{ij}$ denotes a Lorentz metric with signature $(-,-,-,+)$. $R_{ij}$ and $R$ are its Ricci curvature and scalar curvature respectively, and $T_{ij}$ is the so-called stress-energy tensor. When $T_{ij}=0$, the solutions of EFE are called vacuum solutions in physics. Such solutions are called Einstein-Lorentz metrics in mathematics since they have constant Ricci curvature. Almost all the vacuum solutions of EFE in this section can be found in the monograph \cite{MacCallum}.

\begin{example}\label{dinianeingainding}
The first non-trivial vacuum solution of EFE is the well-known Schwarzschild metric
\begin{eqnarray}\label{Schwarzschild1}
\ud s^2=\left(1-\f{2m}{R}\right)\ud t^2-\left(1-\f{2m}{R}\right)^{-1}\ud R^2-R^2(\ud\theta^2+\sin^2\theta\,\ud\varphi^2),
\end{eqnarray}
where $m$ is a constant, $R>2m$. (\ref{Schwarzschild1}) reads in isotropic coordinates by transformation $R=r\left(1+\frac{m}{2r}\right)^2$ as
\begin{eqnarray}\label{Schwarzschild2}
\ud s^2=\f{\left(1-\frac{m}{2r}\right)^2}{\left(1+\frac{m}{2r}\right)^2}\,\ud t^2-\left(1+\frac{m}{2r}\right)^4(\ud x^2+\ud y^2+\ud z^2),
\end{eqnarray}
in which $r:=\sqrt{x^2+y^2+z^2}$. Schwarzschild metric is Ricci-flat, admitting a manifest timelike Killing vector $\partial_t$ for both expressions (\ref{Schwarzschild1}) and (\ref{Schwarzschild2}) since they are invariable under flow $t\rightarrow t+\textrm{constant}$.

Taking $\ba=\sqrt{\ud s^2}$ using (\ref{Schwarzschild2}), then the dual $1$-form of $X=\lambda\partial_t$ ($\lambda$ is a non-zero constant) is $\bb=\lambda\f{\left(1-\frac{m}{2r}\right)^2}{\left(1+\frac{m}{2r}\right)^2}\,\ud t$. $\bar b^2=\lambda^2\f{\left(1-\frac{m}{2r}\right)^2}{\left(1+\frac{m}{2r}\right)^2}$. By (\ref{diaiengnaingiang}), the following two-parameter non-regular Randers metrics
\begin{eqnarray*}
F_{m,\lambda}=\f{\sqrt{\left(1+\frac{m}{2r}\right)^2\left(1-\frac{m}{2r}\right)^2\ud t^2+\left[{\lambda^2\left(1-\frac{m}{2r}\right)^2-\left(1+\frac{m}{2r}\right)^2}\right]\left(1+\frac{m}{2r}\right)^6(\ud x^2+\ud y^2+\ud z^2)}+\lambda\left(1-\frac{m}{2r}\right)^2\ud t}{\lambda^2\left(1-\frac{m}{2r}\right)^2-\left(1+\frac{m}{2r}\right)^2}
\end{eqnarray*}
are Ricci-flat on domain $D_{m,\lambda}$ determined by $\lambda^2\frac{\left(1-\frac{m}{2r}\right)^2}{\left(1+\frac{m}{2r}\right)^2}>1$.

Be attention that $m$ is allowed to be negative,  although it should be positive In General Relativity since it represents the quality of a spherical mass((\ref{Schwarzschild2}) is indeed an Einstein-Lorentz metric for any $m\in\RR$). For instance, when $\lambda=1$, domain $D_{m,0}$ is non-empty only if $m<0$. 

Finally, it is worth to be mentioned that, by applying symmetry transformation $T_\gamma$ introduced in \cite{akbar} on Schwarzschild metric (\ref{Schwarzschild1}), M. Akbar and M. MacCallum obtain one-parameter extended families of Ricci flat metrics as follows,
\begin{eqnarray*}
\ud s^2&=&-r^{2\gamma}\sin^{2\gamma}\theta\left(1-\f{2m}{r}\right)^{\gamma+1}\ud t^2+r^{2\gamma^2-2\gamma}\sin^{2\gamma^2-2\gamma}\theta\left(1-\f{2m}{r}\right)^{\gamma^2+\gamma-1}\ud r^2\nonumber\\
&&+r^{2\gamma^2-2\gamma+2}\sin^{2\gamma^2-2\gamma}\theta\left(1-\f{2m}{r}\right)^{\gamma^2+\gamma}\ud\theta^2
+r^{2-2\gamma}\sin^{2-2\gamma}\theta\left(1-\f{2m}{r}\right)^{-\gamma}\ud\varphi^2.
\end{eqnarray*}
Hence, one can construct more Ricci-flat non-singular Randers metrics by using such generalized  Schwarzschild metrics and their manifest timelike Killing vector $\partial_t$.
\end{example}

\begin{example}
Kerr metric
\begin{eqnarray}\label{Kerr1}
\ud s^2&=&\left(1-\f{2mr}{r^2+a^2\cos^2\theta}\right)(\ud u+a\sin^2\theta\,\ud\varphi)^2-2(\ud u+a\sin^2\theta\,\ud\varphi)(\ud r+a\sin^2\theta\,\ud\varphi)\nonumber\\
&&-(r^2+a^2\cos^2\theta)(\ud\theta^2+\sin^2\theta\,\ud\varphi^2)
\end{eqnarray}
is Ricci-flat, where $m$ and $a$ are constants. $\partial_u$ is a manifest timelike Killing vector when $\f{2mr}{r^2+a^2\cos^2\theta}<1$.

In Cartesian Kerr-Schild coordinates, Kerr metric reads
\begin{eqnarray}\label{Kerr2}
\ud s^2=\ud t^2-\ud x^2-\ud y^2-\ud z^2-\f{2mr^3}{r^4+a^2z^2}\left\{\ud t+\f{r(x\,\ud x+y\,\ud y)}{r^2+a^2}+\f{a(y\,\ud x-x\,\ud y)}{r^2+a^2}+\f{z\,\ud z}{r}\right\}^2,
\end{eqnarray}
where $r$ is determined implicitly by $r^4-(x^2+y^2+z^2-a^2)r^2-a^2z^2=0$. $\partial_t$ is a manifest timelike Killing vector of (\ref{Kerr2}) when $\f{2mr^3}{r^4+a^2z^2}<1$.

Taking $\ba=\sqrt{\ud s^2}$ using (\ref{Kerr2}), then the dual $1$-form of $X=\lambda\partial_t$ is
\begin{eqnarray*}
\bb=\lambda\,\ud t-\f{2\lambda mr^3}{r^4+a^2z^2}\left\{\ud t+\f{r(x\,\ud x+y\,\ud y)}{r^2+a^2}+\f{a(y\,\ud x-x\,\ud y)}{r^2+a^2}+\f{z\,\ud z}{r}\right\}.
\end{eqnarray*}
Obviously, $\bar b^2=\|\partial_t\|_{\ba}^2=\lambda^2\left\{1-\frac{2mr^3}{r^4+a^2z^2}\right\}$. By (\ref{diaiengnaingiang}), the three-parameter non-regular Randers metrics $F_{m,a,\lambda}$ are Ricc-flat on domain $D_{m,a,\lambda}$ determined by $\lambda^2\left\{1-\frac{2mr^3}{r^4+a^2z^2}\right\}>1$. Similar to Example \ref{dinianeingainding}, $m$ is allowed to be negative here.
\end{example}

\begin{example}
C-metric
\begin{eqnarray}\label{C-metric}
\ud s^2=\f{1}{(x+y)^2}\left\{F\,\ud t^2-\f{1}{G}\,\ud x^2-\f{1}{F}\,\ud y^2-G\,\ud z^2\right\}
\end{eqnarray}
is Ricci-flat, where $m$ and $a$ are constants, and $F:=-1+y^2-2may^3>0$, $G:=1-x^2-2max^3>0$. $\partial_t$ is a manifest timelike Killing vector.

Taking $\ba=\sqrt{\ud s^2}$ using (\ref{C-metric}). then the dual $1$-form of $X=\lambda\partial_t$ is $\bb=\frac{\lambda F}{(x+y)^2}\,\ud t$. $\bar b^2=\frac{\lambda^2F}{(x+y)^2}$. By (\ref{diaiengnaingiang}), the following three-parameter non-regular Randers metrics
\begin{eqnarray*}
F_{m,a,\lambda}=\f{\sqrt{(x+y)^2F\,\ud t^2+\left[\lambda^2F-(x+y)^2\right]\left(\frac{1}{G}\,\ud x^2+\frac{1}{F}\,\ud y^2+G\,\ud z^2\right)}+\lambda F\,\ud t}{\lambda^2F-(x+y)^2}
\end{eqnarray*}
are Ricci-flat on domain $D_{m,a,\lambda}$ determined by $\frac{\lambda^2F}{(x+y)^2}>1$.

By applying $T_\gamma$ transformation on C-metric,  M. Akbar and M. MacCallum obtain generalized C-metrics as below,
\begin{eqnarray*}
\ud s^2=\f{F^{\gamma+1}G^\gamma}{(x+y)^{4\gamma+2}}\,\ud t^2-\f{F^{\gamma^2+\gamma}G^{\gamma^2-\gamma-1}}{(x+y)^{4\gamma^2+2}}\,\ud x^2-\f{F^{\gamma^2+\gamma-1}G^{\gamma^2-\gamma}}{(x+y)^{4\gamma^2+2}}\,\ud y^2-\f{F^{-\gamma}G^{-\gamma+1}}{(x+y)^{-4\gamma+2}}\,\ud z^2.
\end{eqnarray*}
They are also Ricci-flat. Hence, one can construct more Ricci-flat non-regular Randers metrics by using such metrics and their manifest timelike Killing vector $\partial_t$.
\end{example}

\begin{example}
Kasner metric
\begin{eqnarray}\label{Kasner}
\ud s^2=\ud t^2-t^{2p}\,\ud x^2-t^{2q}\,\ud y^2-t^{2r}\,\ud z^2,
\end{eqnarray}
with the Kasner exponents $p$, $q$ and $r$ satisfing
\begin{eqnarray*}
p+q+r=1,\qquad p^2+q^2+r^2=1
\end{eqnarray*}
is Ricci-flat, and admitting a homothetic vector
\begin{eqnarray*}
X_1=t\partial_t+(1-p)x\partial_x+(1-q)y\partial_y+(1-r)z\partial_z
\end{eqnarray*}
with homothety factor $1$.

Taking $\ba=\sqrt{\ud s^2}$ using (\ref{Kasner}), then the dual $1$-form of $X_\lambda=\lambda X_1$ is
\begin{eqnarray*}
\bb=\lambda\left\{t\,\ud t-(1-p)xt^{2p}\,\ud x-(1-q)yt^{2q}\,\ud y-(1-r)zt^{2t}\,\ud z\right\}
\end{eqnarray*}
and
\begin{eqnarray*}
\bar b^2=\|X_\lambda\|_{\ba}^2=\lambda^2\left\{t^2-(1-p)^2x^2t^{2p}-(1-q)^2y^2t^{2q}-(1-r)^2z^2t^{2r}\right\}.
\end{eqnarray*}
By (\ref{diaiengnaingiang}), the non-regular Randers metric $F_{p,q,r,\lambda}$ has constant Ricci curvature $-\frac{\lambda^2}{4}$ on domain $D_{p,q,r,\lambda}$ determined by $\bar b^2>1$.
\end{example}

\begin{example}
Levi-Civita metric
\begin{eqnarray}\label{Levi-Civita}
\ud s^2=z^{2\gamma}\,\ud t^2-z^{-2(\gamma-1)}\,\ud x^2-z^{2\gamma(\gamma-1)}(\ud y^2+\ud z^2)
\end{eqnarray}
($\gamma$ is a constant) is Ricci-flat, admitting a manifest timelike Killing vector $\partial_t$. Notice that such metric is flat when $\gamma=0$ or $1$.

Taking $\ba=\sqrt{\ud s^2}$ using (\ref{Levi-Civita}), then the dual $1$-form of $X=\lambda\partial_t$ is
$\bb=\lambda z^{2\gamma}\,\ud t$. $\bar b^2=\lambda^2z^{2\gamma}$. By (\ref{diaiengnaingiang}), the following two-parameter non-regular Randers metrics
\begin{eqnarray*}
F^0_{\gamma,\lambda}=\frac{\sqrt{z^{2\gamma}\,\ud t^2+(\lambda^2z^{2\gamma}-1)z^{-2(\gamma-1)}\,\ud x^2+(\lambda^2z^{2\gamma}-1)z^{2\gamma(\gamma-1)}(\ud y^2+\ud z^2)}+\lambda z^{2\gamma}\,\ud t}{\lambda^2z^{2\gamma}-1}
\end{eqnarray*}
are Ricci-flat on domain $D^0_{\gamma,\lambda}$ determined by $\lambda^2z^{2\gamma}>1$.

Moreover, (\ref{Levi-Civita}) admits a homothetic vector
\begin{eqnarray*}
X_1=\f{1}{\gamma^2-\gamma+1}\left\{(\gamma-1)^2t\partial_t+\gamma^2x\partial_x+y\partial y+z\partial_z\right\},
\end{eqnarray*}
with homothety factor $1$. Obviously, the dual $1$-form of $X_\lambda=\lambda X_1$ is
\begin{eqnarray*}
\bb=\f{\lambda}{\gamma^2-\gamma+1}\left\{(\gamma-1)^2z^{2\gamma}t\,\ud t-\gamma^2z^{-2(\gamma-1)}x\,\ud x-z^{2\gamma(\gamma-1)}(y\,\ud y+z\,\ud z)\right\}
\end{eqnarray*}
and
\begin{eqnarray*}
\bar b^2=\|X_\lambda\|_{\ba}^2=\f{\lambda^2}{(\gamma^2-\gamma+1)^2}\left\{(\gamma-1)^4z^{2\gamma}t^2
-\gamma^4z^{-2(\gamma-1)}x^2-z^{2\gamma(\gamma-1)}(y^2+z^2)\right\}.
\end{eqnarray*}
Hence, $\bb$ is possiblely timelike as long as $\gamma\neq1$. By (\ref{diaiengnaingiang}), one can obtain two-parameter non-regular Randers metrics $F^-_{\gamma,\lambda}$ having constant Ricci curvature $-\frac{\lambda^2}{4}$ on domain $D^-_{\gamma,\lambda}$ determined by $\bar b^2>1$.
\end{example}

\begin{example}
It seems that the known non Ricci-flat exact Einstein-Lorentz metrics are rare. We only fine one in literatures. Cartor-Novotn\'y-Horsk\'y metric
\begin{eqnarray}\label{Cartor1}
\ud s^2=\cos^2\left(\frac{3\sqrt{\Lambda}}{2}z\right)\sin^{-\frac{2}{3}}\left(\frac{3\sqrt{\Lambda}}{2}z\right)\,\ud t^2-\sin^{\frac{4}{3}}\left(\frac{3\sqrt{\Lambda}}{2}z\right)(\ud x^2+\ud y^2)-\ud z^2
\end{eqnarray}
is an Einstein metric with negative Ricci curvature $-\Lambda$. Similarly,
\begin{eqnarray}\label{Cartor2}
\ud s^2=\cosh^2\left(\frac{3\sqrt{\Lambda}}{2}z\right)\sinh^{-\frac{2}{3}}\left(\frac{3\sqrt{\Lambda}}{2}z\right)\,\ud t^2-\sinh^{\frac{4}{3}}\left(\frac{3\sqrt{\Lambda}}{2}z\right)(\ud x^2+\ud y^2)-\ud z^2
\end{eqnarray}
has positive Ricci curvature $\Lambda$. Both of them admit a manifest timelike Killing vector $\partial_t$.

Taking $\ba=\sqrt{\ud s^2}$ using (\ref{Cartor1}) with $\Lambda=\frac{4}{9}$, then the dual $1$-form of $X=\lambda\partial_t$ is $\bb=\lambda\cos^2z\sin^{-\frac{2}{3}}z\,\ud t$. $\bar b^2=\lambda^2\cos^2z\sin^{-\frac{2}{3}}z$. By (\ref{diaiengnaingiang}), the following non-regular Randers metric
\begin{eqnarray*}
F=\f{{\sqrt{\cos^2z\sin^{\frac{2}{3}}z\,\ud t^2+\left(\lambda^2\cos^2z-\sin^{\frac{2}{3}}z\right)
\sin^{\frac{2}{3}}z\left[\sin^{\frac{4}{3}}z(\ud x^2+\ud y^2)+
\ud z^2\right]}+\lambda\cos^2z\,\ud t}}{\lambda^2\cos^2z-\sin^{\frac{2}{3}}z}
\end{eqnarray*}
has constant Ricci curvature $-\frac{4}{9}$ on domain determined by $\lambda^2\cos^2z\sin^{-\frac{2}{3}}z>1$, which is a strip domain
\begin{eqnarray*}
0<z<\arcsin\sqrt{1-\frac{1}{6}\lambda^{-3}\left(108\lambda^3+12\sqrt{81\lambda^6+12}\right)^\frac{1}{3}
+2\lambda^{-3}\left(108\lambda^3+12\sqrt{81\lambda^6+12}\right)^{-\frac{1}{3}}}.
\end{eqnarray*}
One can obtain some non-regular Einstein-Randers metrics of positive Ricci curvature by using (\ref{Cartor2}).
\end{example}

There are more exact vacuum solutions of EFE, such as
\begin{itemize}
\item Taub-NUT metric:
\begin{eqnarray*}\label{Taub-NUT}
\ud s^2=-f(r)(\ud t+2l\cos\theta\,\ud\varphi)^2+f^{-1}(r)\,\ud r^2+(r^2+l^2)(\ud\theta^2+\sin^2\theta\,\ud\varphi^2),
\end{eqnarray*}
where $f(r)=\frac{r^2-2mr-l^2}{r^2+l^2}$~($m$, $l$ are constants).
\item
Kerns-Wild metric:
\begin{eqnarray*}
\ud s^2&=&-\left(1-\f{2m}{r}\right)e^{2c(m-r)\cos\theta}\,\ud t^2+e^{2c(r-3m)\cos\theta-c^2(r^2-2mr)\sin^2\theta}\left\{\left(1-\f{2m}{r}\right)^{-1}\ud r^2+r^2\,\ud\theta^2\right\}\nonumber\\
&&+r^2\sin^2\theta e^{-2c(m-r)\cos\theta}\,\ud\varphi^2.
\end{eqnarray*}
\item
Oszv\'{a}th-Sch\"{u}cking's "anti-Mach" metric:
\begin{eqnarray*}
\ud s^2=\ud x^2+\ud y^2-2\,\ud u\ud v-2\left\{(x^2-y^2)\cos(2u)-2xy\sin(2u)\right\}\,\ud u^2.
\end{eqnarray*}
\end{itemize}
All of them are Ricci-flat and admit non-trivial timelike homothetic vector fields. One can obtain more non-regular Einstein-Randers metrics by a similar process.

\section{Singular Randers metrics}
We will call a Randers metric with $b\equiv1$ a \emph{singular Randers metrics}. This is a very special case.

According to Theorem \ref{uabubdagnianengga}, when $b\neq1$, the structure of Randers metric with constant flag curvature is similar to that with constant Ricci curvature. However, it seems that there is a gap between the singular Randers metrics with constant flag curvature and those with Ricci curvature.  Specifically speaking, if a singular Randers metric $F$ is of constant flag curvature, then (\ref{yawbbwjbabgg}) holds with the factor $c$ being a constant, but if $F$ is of constant Ricci curvature, $c$ is \emph{not necessary} a constant.

At present, we can't provide a more concise description for singular Randers metrics even with constant flag curvature. However, the second author have found some exact examples in 2015 with another coauthor.

Due to Example 9.6 in \cite{yct-zhm-pfga}, we shown that if $\ba$ is a Riemann metric with constant sectional curvature $\bar\mu$, and $\bb$ is closed and conformal with respect to $\ba$, with the conformal factor $c(x)$ satisfying $c^2=-\bar\mu\bar b^2$, then the general $\ab$-metric
\begin{eqnarray}\label{ainianingnangg}
F=\f{\sqrt{(1-\bar b^2)\ba^2+\bb^2}}{1-\bar b^2}-\f{\bb}{\bar b(1-\bar b^2)}.
\end{eqnarray}
has constant flag curvature $K=-\frac{1}{4}$.

Exist data $(\ba,\bb)$ satisfying the required conditions. Taking
\begin{eqnarray}\label{ab}
\ba=\f{\sqrt{(1+\bar\mu|x|^2)|y|^2-\bar\mu\langle x,y\rangle^2}}{1+\bar\mu|x|^2},\qquad
\bb=\f{\lambda\langle x,y\rangle+(1+\bar\mu|x|^2)\langle a,y\rangle-\bar\mu\langle a,x\rangle\langle x,y\rangle}{(1+\bar\mu|x|^2)^\frac{3}{2}},
\end{eqnarray}
then $\ba$ has constant curvature $\bar\mu$, and $\bb$ is closed and conformal with respect to $\ba$, with the conformal factor $c(x)$ satisfying
\begin{eqnarray*}
c^2=\lambda^2+\bar\mu|a|^2-\bar\mu\bar b^2.
\end{eqnarray*}
So, if the constant number $\lambda$ and the constant vector $a$ satisfy $\lambda^2+\bar\mu|a|^2=0$, then the corresponding data $(\ba,\bb)$ is what we need. It is interesting that it occurs only for hyperbolic metrics.

Take $\a=\frac{\sqrt{(1-\bar b^2)\ba^2+\bb^2}}{1-\bar b^2}$ and $\b=-\frac{\bb}{\bar b(1-\bar b^2)}$. One can verify that $\|\b\|_{\a}\equiv1$. That is to say, (\ref{ainianingnangg}) is a singular Randers metric in fact. The above exact metrics are the first examples of singular Randers metrics with constant flag curvature.

Singular Randers metrics are not the solution of any kind of navigation problem on Riemann or Lorentz spaces, since you can't obtain a elliptical or hyperbolic hypersurface forever by shifting a parabolic hypersurface. This is a monstrous disaster. In other words, just because $b\equiv1$, the crucial navigation deformations (\ref {donaienggainiga}) for strongly convex Randers metrics are invalid for singular Randers metrics.

It is amazing that the expression (\ref{ainianingnangg}) is very similar to the navigation expression (\ref{dnaienignadng}) for strongly convex Randers metrics. At the very beginning, we thank maybe such expression will become the key to reveal the curvature structure of singular Einstein-Randers metrics. We guessed that a singular Randers metric (\ref{ainianingnangg}) is an Einstein metric if and only if $\ba$ is an Einstein metric and $\b$ is conformal with respect to $\bb$ (with some possible additional conditions), just like the strongly convex Randers metrics. Unfortunately, it is not true.

Gradually, we realized that the structure of singular Einstein-Randers metrics is extremely complicated. In fact, the second author plan to write a series of papers to demonstrate this theme, and we believe the key method is the so-called $\b$-deformations, which generalize the navigation deformations for strongly convex Randers metrics in a natural way\cite{yct-dhfp}.

\noindent Xiaoyun Tang\\
School of Mathematical Sciences, South China Normal
University, Guangzhou, 510631, P.R. China\\
tangxy15@fudan.edu.cn
\newline
\newline
\newline
\noindent Changtao Yu\\
School of Mathematical Sciences, South China Normal
University, Guangzhou, 510631, P.R. China\\
aizhenli@gmail.com
\end{document}

%% file: symbol.tex
\newcommand{\ud}{\mathrm{d}}
\newcommand{\RR}{\mathbb{R}}
\newcommand{\f}{\frac}

\newcommand{\xx}{|x|^2}
\newcommand{\yy}{|y|^2}
\newcommand{\xy}{\langle x,y\rangle}

\newcommand{\pppp}[4]%
  {\frac{\partial^3{#1}}{\partial{#2}\partial{#3}\partial{#4}}}

\newcommand{\yI}{y^i}
\newcommand{\yk}{y_k}

\renewcommand{\a}{\alpha}
\renewcommand{\b}{\beta}
\newcommand{\ab}{(\alpha,\beta)}

\newcommand{\ba}{\bar\alpha}
\newcommand{\bb}{\bar\beta}

\newcommand{\aij}{a_{ij}}

\newcommand{\bij}{b_{i|j}}

\newcommand{\baij}{\bar a_{ij}}

\newcommand{\bbij}{\bar b_{i|j}}
\newcommand{\bi}{b_i}
\newcommand{\bj}{b_j}
\newcommand{\bk}{b_k}
\newcommand{\bI}{b^i}
\newcommand{\bK}{b^k}

\newcommand{\rij}{r_{ij}}

\newcommand{\roo}{r_{00}}

\newcommand{\rIiho}{r^i{}_{i|0}}
\newcommand{\rIohi}{r^i{}_{0|i}}
\newcommand{\rIihb}{r^i{}_{i|b}}
\newcommand{\rooho}{r_{00|0}}
\newcommand{\roohk}{r_{00|k}}

\newcommand{\rkoho}{r_{k0|0}}
\newcommand{\rkohb}{r_{k0|b}}
\newcommand{\rIhk}{r^i{}_{|k}}
\newcommand{\rIho}{r^i{}_{|0}}
\newcommand{\rIhi}{r^i{}_{|i}}

\newcommand{\roho}{r_{0|0}}
\newcommand{\rohk}{r_{0|k}}
\newcommand{\rkho}{r_{k|0}}

\newcommand{\rohb}{r_{0|b}}

\newcommand{\rhk}{r_{|k}}

\newcommand{\rIkho}{r^i{}_{k|0}}
\newcommand{\rIkhb}{r^i{}_{k|b}}
\newcommand{\rIohk}{r^i{}_{0|k}}
\newcommand{\rkohI}{r_{k0|}{}^{i}}
\newcommand{\rkhI}{r_{k|}{}^{i}}
\newcommand{\rIoho}{r^i{}_{0|0}}
\newcommand{\roohI}{r_{00|}{}^{i}}
\newcommand{\rIohb}{r^i{}_{0|b}}
\newcommand{\rohI}{r_{0|}{}^{i}}

\newcommand{\sij}{s_{ij}}
\newcommand{\sIk}{s^i{}_k}
\newcommand{\sIo}{s^i{}_0}
\newcommand{\sko}{s_{k0}}

\newcommand{\si}{s_i}
\newcommand{\sj}{s_j}
\newcommand{\sk}{s_k}
\newcommand{\sI}{s^i}

\newcommand{\so}{s_0}

\newcommand{\sIohi}{s^i{}_{0|i}}

\newcommand{\sIohk}{s^i{}_{0|k}}
\newcommand{\sIohb}{s^i{}_{0|b}}
\newcommand{\sIkho}{s^i{}_{k|0}}

\newcommand{\sIkhi}{s^i{}_{k|i}}

\newcommand{\sIoho}{s^i{}_{0|0}}
\newcommand{\sIhi}{s^i{}_{|i}}

\newcommand{\sIhk}{s^i{}_{|k}}
\newcommand{\skhI}{s_{k|}{}^{i}}
\newcommand{\sIho}{s^i{}_{|0}}
\newcommand{\soho}{s_{0|0}}
\newcommand{\sohb}{s_{0|b}}

\newcommand{\sohk}{s_{0|k}}

\newcommand{\skho}{s_{k|0}}

\newcommand{\co}{c_0}
\newcommand{\cb}{c_b}

\newcommand{\ck}{c_k}

\newcommand{\poo}{p_{00}}

\newcommand{\pk}{p_k}
\newcommand{\po}{p_0}
\newcommand{\pIi}{p^i{}_i}
\newcommand{\pIk}{p^i{}_k}
\newcommand{\pko}{p_{k0}}

\newcommand{\qk}{q_k}

\newcommand{\qIo}{q^i{}_0}
\newcommand{\qoo}{q_{00}}
\newcommand{\qoI}{q_{0}{}^i}
\newcommand{\qko}{q_{k0}}
\newcommand{\qok}{q_{0k}}

\newcommand{\qqk}{q^*{}_k}

\newcommand{\qqo}{q^*{}_0}
\newcommand{\qIi}{q^i{}_{i}}

\newcommand{\too}{t_{00}}

\newcommand{\tIo}{t^i{}_{0}}
\newcommand{\tko}{t_{k0}}
\newcommand{\tk}{t_k}

\newcommand{\tI}{t^i}
\renewcommand{\to}{t_0}
\newcommand{\tIk}{t^i{}_{k}}
\newcommand{\tIi}{t^i{}_{i}}

\newcommand{\dIk}{\delta^i{}_k}

\newcommand{\RIkF}{{}^F{R}^i{}_k}
\newcommand{\RIk}{{R}^i{}_k}

\newcommand{\RIb}{{R}^i{}_b}
\newcommand{\Rbb}{R_{bb}}

\newcommand{\RbIkb}{R_{b}{}^i{}_{kb}}
\newcommand{\RoIkb}{R_{0}{}^i{}_{kb}}
\newcommand{\RbIko}{R_{b}{}^i{}_{k0}}
\newcommand{\Rbokb}{R_{b0kb}}

\newcommand{\RicooF}{{}^F{Ric_{00}}}
\newcommand{\Ricoo}{{Ric_{00}}}

\newcommand{\Rico}{{Ric_{0}}}

\newcommand{\Ric}{{Ric}}

\newcommand{\Rat}{\mathfrak{Rat}}
\newcommand{\Irrat}{\mathfrak{Irrat}}
\newcommand{\rat}{{\mathfrak{rat}^i{}_k}}
\newcommand{\irrat}{{\mathfrak{irrat}^i{}_k}}